\documentclass[11pt]{amsart}
\usepackage{hyperref}
\usepackage{amsfonts}
\usepackage{amsmath}
\usepackage{amssymb}
\usepackage{amscd}
\usepackage{graphicx}
\usepackage{enumitem}
\usepackage{latexsym}
\usepackage{color}
\usepackage{comment}
\usepackage{epsfig}
\usepackage{amsopn}
\usepackage{xypic}
\usepackage{mathrsfs}
\usepackage{array}
\usepackage[usenames,dvipsnames]{pstricks}
\usepackage{epsfig}
\usepackage{pst-grad} 
\usepackage{pst-plot} 
\usepackage[space]{grffile} 
\usepackage{etoolbox} 
\makeatletter 
\patchcmd\Gread@eps{\@inputcheck#1 }{\@inputcheck"#1"\relax}{}{}
\makeatother

\hypersetup{pdfpagemode=UseNone}
\usepackage{booktabs}
\usepackage{longtable}
\theoremstyle{plain}
\newtheorem{lemma}{Lemma}[section]
\newtheorem*{theorem*}{Theorem}
\newtheorem*{lemma*}{Lemma}
\newtheorem*{proposition*}{Proposition}
\newtheorem*{conjecture*}{Conjecture}
\newtheorem*{corollary*}{Corollary}
\newtheorem*{problem*}{Problem}
\newtheorem{theorem}[lemma]{Theorem}

\newtheorem{corollary}[lemma]{Corollary}
\newtheorem{proposition}[lemma]{Proposition}

\newtheorem{question}[lemma]{Question}

\theoremstyle{definition}

\newtheorem{example}[lemma]{Example}
\newtheorem{remark}[lemma]{Remark}

\newcommand{\F}[1]{\mathscr{#1}}
\newcommand{\fto}[1]{\stackrel{#1}{\to}}

\newcommand{\Z}{\mathbb{Z}}

\newcommand{\FF}{\mathbb{F}}
\renewcommand{\F}{\mathbb{F}}
\newcommand{\CC}{\mathbb{C}}
\newcommand{\QQ}{\mathbb{Q}}

\newcommand{\OO}{\mathcal{O}}
\newcommand{\te}{\otimes}

\newcommand{\cM}{\mathcal M}

\newcommand{\cP}{\mathcal P}

\newcommand{\bv}{{\bf v}}
\newcommand{\bu}{{\bf u}}

\newcommand{\cV}{\mathcal{V}}

\newcommand{\leqpar}{\underset{{\scriptscriptstyle (}-{\scriptscriptstyle )}}{<}}

\renewcommand{\P}{\mathbb{P}}

\newcommand{\PP}{\mathbb{P}}

\DeclareMathOperator{\Flag}{Flag}

\DeclareMathOperator{\ch}{ch}

\DeclareMathOperator{\Hom}{Hom}

\DeclareMathOperator{\Pic}{Pic}

\DeclareMathOperator{\rk}{rk}

\DeclareMathOperator{\Ext}{Ext}
\DeclareMathOperator{\ext}{ext}

\DeclareMathOperator{\Coh}{Coh}
\DeclareMathOperator{\sHom}{\mathcal{H}\kern -.5pt\mathit{om}}
\DeclareMathOperator{\sTor}{\mathcal{T}\kern -1.5pt\mathit{or}}

\DeclareMathOperator{\codim}{codim}

\begin{document}

\date{\today}

\author[J. Huizenga]{Jack Huizenga}
\address{Department of Mathematics, The Pennsylvania State University, University Park, PA 16802}
\email{huizenga@psu.edu}

\author[J. Kopper]{John Kopper}
\address{Department of Mathematics, The Pennsylvania State University, University Park, PA 16802}
\email{kopper@psu.edu}

\thanks{During the preparation of this article the first author was partially supported by the NSF FRG grant   DMS 1664303}

\title{Ample stable vector bundles on rational surfaces}

\begin{abstract}
We study ample stable vector bundles on minimal rational surfaces. We give a complete classification of those moduli spaces for which the general stable bundle is both ample and globally generated. We also prove that if $V$ is any stable bundle, then a large enough direct sum $V^{\oplus n}$  has  ample deformations unless there is an obvious numerical reason why it cannot.   Previous work in this area  has mostly focused on rank two bundles and relied primarily on classical constructions such as the Serre construction.  In contrast, we use recent advances in moduli of vector bundles to obtain strong results for vector bundles of any rank.

\end{abstract}

\maketitle

\setcounter{tocdepth}{1}
\tableofcontents

\section{Introduction} On a curve, a stable vector bundle is ample if and only if it has positive degree. In higher dimensions, Fulton \cite{Fulton} shows ampleness cannot be characterized numerically, even for stable bundles. Fulton and Lazarsfeld \cite{FultonLazarsfeld} give necessary numerical conditions for the existence of an ample bundle on any projective variety, but on any particular variety these conditions are typically far from sufficient.  A classification of ample bundles has remained elusive even in the case of very well-known surfaces.  In this paper, we work towards classifying the numerical invariants of ample stable bundles on minimal rational surfaces.  Previous work in this area  has mostly focused on rank two bundles and relied primarily on classical constructions such as the Serre construction (see, for example, \cite{BallicoLanteri}, \cite{LePotierAmple}, \cite{Sterian}).  In contrast, we use recent advances in moduli of vector bundles to obtain strong results for vector bundles of any rank.

Because ampleness is open in families, there exists a stable ample vector bundle of Chern character $\bv$ if and only if the general stable bundle in the same component of the moduli space $M(\bv)$ of semistable sheaves is also ample. On minimal rational surfaces, the moduli spaces $M(\bv)$ are always irreducible, suggesting the following natural rephrasing of the classification problem.
\begin{question}\label{q:intro}
For which Chern characters $\bv$ on a minimal rational surface is the general  bundle in $M(\bv)$ ample?
\end{question}
We begin by discussing the known numerical obstructions to ampleness, and then state our main theorems which construct ample bundles to investigate how sharp these obstructions are.

\subsection{Obstructions to ampleness}  There are two main known numerical obstructions to ampleness.  The main theorem of \cite{FultonLazarsfeld} implies that if $V$ is an ample bundle on a smooth surface then  
\begin{equation}\label{FLbound}\tag{$\ast$}
\frac{1}{2}\nu(V)^2 > \frac{\Delta(V)}{\rk(V) + 1},
\end{equation}
where $$\nu(V) = \frac{c_1(V)}{\rk(V)} \qquad \textrm{and}\qquad \Delta(V) = \frac{1}{2}\nu(V)^2 - \frac{\ch_2(V)}{\rk(V)}$$ are the \emph{total slope} and \emph{discriminant}, respectively.  These are the so-called \emph{logarithmic invariants} and remain constant when a Chern character is scaled.

The curves on a particular surface provide additional obstructions to ampleness beyond the ``universal'' Fulton-Lazarsfeld inequality (\ref{FLbound}).  An ample bundle must be ample when restricted to any curve.  Let $\PP^2$ be the projective plane with hyperplane class $H$, and let $\FF_e= \PP(\OO_{\PP^1}\oplus\OO_{\PP^1}(e))\to \P^1$ be the Hirzebruch surface with fiber class $F$, section $E$ of self-intersection $-e$, and minimal  polarization $H = E + (e+1)F$.  Analyzing the restriction of a stable vector bundle to these curves, we find the following further necessary conditions for ampleness.

\begin{proposition*}[\ref{prop:obstruction}]
Let $X=\P^2$ or $\F_e$ and let $V$ be an ample stable vector bundle on $X$ of rank at least $2$.
\begin{enumerate}
        \item If $X = \PP^2$, then $\nu(V) \cdot H > 1 + \frac{1}{\rk(V)}$ or $V\cong T_{\P^2}$.
\item If $X = \FF_0$, then $\nu(V)\cdot F > 1$ and $\nu(V)\cdot E > 1$.
        \item If $X = \FF_e$ with $e\geq 1$, then $\nu(V) \cdot F > 1$ and $\nu(V) \cdot E \geq 1$.
\end{enumerate}  
\end{proposition*}

In general, the inequalities in Proposition \ref{prop:obstruction} do not imply the Fulton-Lazarsfeld bound (\ref{FLbound}) since they do not involve the discriminant $\Delta(V)$.   Thus they are not sufficient to guarantee ampleness.   Nevertheless, the two main results in this paper show that under some additional assumptions, the conditions of Proposition \ref{prop:obstruction} are sufficient to guarantee that the general stable bundle is ample.

\subsection{Asymptotic ampleness} In our first approach, we consider what happens when we replace a Chern character $\bv$ with a sufficiently large multiple $n\bv$.  Observe that if $\nu(V)$ and $\Delta(V)$ are held fixed but $\rk(V)$ becomes arbitrarily large, then the inequality (\ref{FLbound}) reduces to the simpler condition $\nu(V)^2 > 0$.  This inequality is easily implied by the necessary inequalities of Proposition \ref{prop:obstruction}, so it no longer provides an obstruction to ampleness.  Our first main theorem shows there are no further obstructions.

\begin{theorem*}[\ref{thm:asymptoticample}]
Let $X = \P^2$ or $\F_e$ and let $\bv$ be the Chern character of a stable vector bundle. Suppose either
\begin{enumerate}
\item $X = \P^2$ and $\nu(\bv)\cdot H > 1$, or

\item $X = \F_0$ and $\nu(\bv)\cdot F > 1$ and $\nu(\bv) \cdot E > 1$, or  

\item $X = \F_e$ with $e\geq 1$ and $\nu(\bv) \cdot F > 1$ and $\nu(\bv) \cdot E \geq 1$. 
\end{enumerate}
If $n$ is a sufficiently large integer and $V\in M(n\bv)$ is general, then $V$ is ample.
\end{theorem*}

The required inequalities in the theorem can be more compactly rephrased as the condition that $\nu(\bv)-H$ is big and nef.  An effective bound on $n$ can be easily given; see Remark \ref{rem-effective} for details.  The theorem also easily implies the next striking corollary which was stated in the abstract.

\begin{corollary}
Let $X=\P^2$ or $\F_e$ and suppose $V$ is a stable bundle such that $\nu(V)$ satisfies the inequalities in Theorem \ref{thm:asymptoticample}.  Then any sufficiently large direct sum $V^{\oplus n}$ has ample semistable deformations.
\end{corollary}

The proof of Theorem \ref{thm:asymptoticample} requires new construction techniques for stable ample bundles.  One of the easiest ways to construct an ample bundle is to take a quotient of a known ample bundle.  However, in practice it is very difficult to verify the stability of such a bundle.  Our key idea is to construct an ample bundle of the desired Chern character $n\bv$ that is not obviously stable, but does have stable deformations.  Since ampleness is an open condition, it follows that the general stable bundle is ample.  Our construction of an appropriate ample quotient relies on the recent classification of moduli spaces where the general bundle is globally generated.  To show that our bundle has stable deformations, we use recent results computing the cohomology of the general bundle in a moduli space.   We summarize these necessary preliminary results from \cite{CoskunHuizengaBN} in \S\ref{sec:BN-prelim} and \S\ref{sec:gg-prelim}.  These problems are being actively researched by several authors on other surfaces (e.g., K3 surfaces), and we expect our techniques to be useful for constructing ample bundles on those surfaces as well.   

\begin{example}
On $\P^2$, if the rank is $\rk(V) = 2$ and the total slope is $\frac{3}{2}H$, then the inequality (\ref{FLbound}) reads $\Delta(V) < \frac{27}{8}$.  However, the only ample stable bundle with this rank and slope is the tangent bundle with discriminant $\Delta = \frac{3}{8}$ (see Proposition \ref{prop:converse_notample_p2}).  Thus (\ref{FLbound}) is far from sharp in this case. In particular, even though the character $\bv = (r,\nu,\Delta) = (2,\frac{3}{2}H,\frac{7}{8})$ satisfies the Fulton-Lazarsfeld bound (\ref{FLbound}) and the necessary conditions of Proposition \ref{prop:obstruction}, there are no stable ample bundles of character $\bv$.  However, Theorem \ref{thm:asymptoticample} shows that there are stable ample bundles of character $n\bv$ for sufficiently large $n$.
\end{example}

\subsection{Ample globally generated bundles} Next, we consider the ampleness problem under the additional assumption that the general sheaf $V\in M(\bv)$ is globally generated.  Characters $\bv$ with this property were completely classified in \cite{CoskunHuizengaBN}, and can be described by certain inequalities; see \S\ref{sec:gg-prelim}.  In this case, Gieseker's Lemma allows us to check ampleness by restricting to curves.

\begin{theorem*}[\ref{thm:main-gg}]
Let $X = \P^2$ or $\F_e$ and let $\bv$ be a Chern character such that the general sheaf $V\in M(\bv)$ is a globally generated vector bundle.  Suppose either
\begin{enumerate}
\item $X = \P^2$ and $\mu(\bv) > 1+\frac{1}{\rk(\bv)}$, or

\item $X = \F_0$ and $\nu(\bv)\cdot F > 1$ and $\nu(\bv) \cdot E > 1$, or  

\item $X = \F_e$ with $e\geq 1$ and $\nu(\bv) \cdot F > 1$ and $\nu(\bv) \cdot E \geq 1$. 
\end{enumerate}
Then $V$ is ample.
\end{theorem*}

We sketch an outline of the proof.  First we identify a list of curve classes that could potentially cause the failure of ampleness.  The irreducible curves of these classes turn out to be smooth rational curves, and we prove the restriction maps from the moduli space $M(\bv)$ to the stack $\Coh(\P^1)$ are smooth.  This allows us to do a dimension count to show that any given curve class cannot cause the failure of ampleness.  Our results are similar in spirit to the work of Le Potier on $\P^2$ in the rank two case \cite{LePotierAmple}; however, the recent computation of the cohomology of a general sheaf in \cite{CoskunHuizengaBN} is crucial to extending the story to arbitrary rank.

\begin{remark}
The full solution to Question \ref{q:intro} remains open.  Gieseker \cite{GiesekerAmple} gives the following interesting example on $\P^2$:  let  $V_{d}$ be a general cokernel of the form $$0\to \OO_{\P^2}(-d)^{2}\to \OO_{\P^2}(-1)^{4} \to V_d\to 0.$$ If $d\leq 6$  then $V_{d}$ does not satisfy (\ref{FLbound}), so $V_d$ is not ample.  But Gieseker shows that $V_d$ has stable deformations and is ample if $d\gg 0$.  
Determining an explicit $d$ such that $V_d$ is ample is an interesting open problem.  However, we  show that if $d\geq 12$ then $V_d \oplus V_d$ has ample stable deformations (see Example \ref{example:Gieseker}).
\end{remark}

\subsection*{Structure of the paper}
In \S \ref{sec:preliminaries}, we recall the necessary background on moduli of sheaves, the Weak Brill-Noether theorems, and the classifications of globally generated characters.  In \S\ref{sec:notample}, we give our necessary conditions for ampleness.  We prove our two main theorems in \S \ref{sec:asymptotics} and \S \ref{sec:ample_globallygenerated}.

\subsection*{Acknowledgments}
We would like to thank Izzet Coskun for many valuable discussions.

\section{Preliminaries}\label{sec:preliminaries}

\subsection{Numerical invariants and semistability} We refer the reader to \cite{HuybrechtsLehn} and \cite{LePotier} for further information on stability and moduli spaces of sheaves.  Let $X$ be a smooth surface with a polarization $H$, and let $\bv\in K(X)$ be a Chern character of nonzero rank. We define the \emph{slope} $\mu(\bv)$, the \emph{total slope} $\nu(\bv)$, and the \emph{discriminant} $\Delta(\bv)$ by the formulas
\[
\mu(\bv) = \frac{c_1(\bv)\cdot H}{\rk(\bv) H^2}, \qquad \nu(\bv) = \frac{c_1(\bv)}{\rk(\bv)},\qquad \textrm{and}  \qquad \Delta(\bv) = \frac{1}{2}\nu(\bv)^2 - \frac{\ch_2(\bv)}{\rk(\bv)}.
\]
For a torsion-free sheaf $V$, we define the invariants $\mu(V)$, $\nu(V)$, $\Delta(V)$ by first taking the Chern character.   We define the \emph{Hilbert polynomial} $P_V(m)$ and the \emph{reduced Hilbert polynomial} $p_V(m)$ by $$P_V(m) = \chi(V(mH)) \qquad\textrm{and}\qquad p_V(m) = \frac{\chi(V(mH))}{\rk(V)}.$$  In terms of these invariants, the Riemann-Roch formula reads $$\chi(V) = \rk(\bv) (P(\nu(\bv))-\Delta(\bv)),$$ where $$P(\nu) =  \chi(\OO_X)+ \frac{1}{2}(\nu^2-\nu\cdot K_X)$$ is the Hilbert polynomial of $\OO_X$.

A torsion-free sheaf $V$ is \emph{$H$-Gieseker (semi)stable} if for all proper subsheaves $W \subset V$, we have $p_W(m) \leqpar p_V(m)$ for $m \gg 0$.  The Bogomolov inequality states that if $V$ is semistable then $\Delta(V) \geq 0$.  If $V$ and $W$ are $H$-semistable and $\mu(V)>\mu(W)$, then $\Hom(V,W) = 0$.  For $\bv\in K(X)$, there are projective moduli spaces $M(\bv) = M_{X,H}(\bv)$ of $H$-Gieseker semistable sheaves on $X$.  

\subsection{The projective plane and Hirzebruch surfaces}  We will primarily be interested in the surfaces $\P^2$ and $\F_e$, so we fix some notation to streamline the exposition.  See \cite{Beauville} or \cite{Hartshorne} for additional information on these surfaces.  On $\P^2$, we write $H$ for the class of a line and also write $L = H$.  The Picard group is $\Pic(\P^2) = \Z H$, and the canonical class is $K_{\P^2} = -3H$.  The polynomial $P$ in the Riemann-Roch formula is given by $$P(x) = \frac{1}{2}(x^2+3x+2).$$

For $e\geq 0$, let  $\F_e = \P(\OO_{\P^1}\oplus \OO_{\P^1}(e))$ be the Hirzebruch surface with a section $E$ of self-intersection $E^2 = -e$.  We let $\pi:\F_e\to \P^1$ be the natural projection, and let $F$ be the class of a fiber.  We let $H = E+(e+1)F$ be the minimal ample polarization, and let $L = F$ be the class of a fiber.  We have $\Pic(\F_e) = \Z E \oplus \Z F$  and the intersection numbers are given by 
\[
F^2 = 0, \quad E^2 = -e, \quad E \cdot F = 1.
\]
The effective cone of divisors is spanned by $E$ and $F$, and the nef cone is spanned by $E+eF$ and $F$.  The canonical class is $K_{\F_e} = -2E-(e+2)F$.  The polynomial $P$ in the Riemann-Roch formula can be written as $$P(xE+yF) = (x+1)\left(y+1-\frac{1}{2}ex\right).$$

\subsection{Ample vector bundles}
A vector bundle $V$ on a projective variety $X$ is said to be \emph{ample} if $\OO_{\P V}(1)$ is an ample line bundle. We collect a few well-known facts about ample bundles in the following proposition 

\begin{proposition}[\cite{LazarsfeldPositivityVol2}]
Let $X$ be a projective variety.
\begin{enumerate} 
\item Any quotient of an ample bundle on $X$ is ample.
\item The direct sum $V \oplus W$ of two bundles $V$ and $W$ is ample if and only if $V$ and $W$ are ample.
\item If $V$ is ample and $W$ is nef, then $V\te W$ is ample.
\item Ampleness of vector bundles on $X$ is open in families.
\end{enumerate}
\end{proposition}

In the final section of the paper we will use Gieseker's Lemma to analyze ampleness of globally generated bundles.

\begin{lemma}[Gieseker's Lemma \protect{\cite[Prop. 6.1.7]{LazarsfeldPositivityVol2}}]
Let $V$ be a globally generated vector bundle on an irreducible projective variety $X$.  Then $V$ fails to be ample if and only if there exists an irreducible curve $D \subset X$ such that $V|_D$ has a trivial quotient.
\end{lemma}

\subsection{Prioritary sheaves} Let $D$ be a divisor class on the smooth surface $X$. A sheaf $V$ is called \emph{$D$-prioritary} if it is torsion-free and satisfies $\Ext^2(V,V(-D)) = 0$. We denote by $\mathcal{P}_{D}(\bv)$ the stack of $D$-prioritary sheaves on $X$ with Chern character $\bv$.

A smooth family of sheaves $\mathcal{V}_s/S$ is called \emph{complete} if the Kodaira-Spencer map $T_sS \to \Ext^1(\mathcal{V}_s,\mathcal{V}_s)$ is surjective for all $s$ in $S$. The importance of $D$-prioritary sheaves is that they induce complete families on restrictions to curves of class $D$. Explicitly, if $\mathcal{V}_s/S$ is a complete family of $D$-prioritary locally free sheaves on $X$, then we obtain a family $\mathcal{V}_s|_D/S$ on $D$, and the Kodaira-Spencer map is the composition
\[
T_sS \to \Ext^1(\mathcal{V}_s,\mathcal{V}_s) \to \Ext^1(\mathcal{V}_s|_D,\mathcal{V}_s|_D).
\]
The first map is surjective because the family $\mathcal{V}_s/S$ is assumed to be complete. To see that the second map surjects, apply the functor $\Hom(\mathcal{V}_s,-)$ to the restriction exact sequence
\[
0 \to \mathcal{V}_s(-D) \to \mathcal{V}_s \to \mathcal{V}_s|_D \to 0.
\] to get an exact sequence $$\Ext^1(\cV_s,\cV_s) \to \Ext^1(\cV_s,\cV_s|_D)\to \Ext^2(\cV_s,\cV_s(-D)).$$ We can identify $\Ext^1(\cV_s,\cV_s|_D)$ with $\Ext^1(\cV_s|_D,\cV_s|_D)$, and since $\cV_s$ is prioritary we find that the map $\Ext^1(\cV_s,\cV_s)\to \Ext^1(\cV_s|_D,\cV_s|_D)$  is surjective.

When $V$ is an $H$-semistable sheaf, it is automatically $D$-prioritary for any divisor $D$ such that $(K_X+D)\cdot H <0$.  Indeed, by Serre duality $$\Ext^2(V,V(-D)) \cong \Hom(V,V(K_X+D))^\ast$$ so $\Hom(V,V(K_X+D))=0$ by stability.  

In particular, when $X = \P^2$ or $\F_e$, we see that $H$-semistable sheaves are $L$-prioritary.  Therefore the stack $\cM(\bv)$ of semistable sheaves is an open substack of $\cP_L(\bv)$.  In particular, to check that an open property holds for the general semistable sheaf, it is sufficient to produce a single prioritary sheaf with that property.  Much is known about these moduli stacks.  We summarize a few important properties here.  
 
\begin{theorem}[\cite{HirschowitzLaszlo} in the $\P^2$ case; \cite{Walter} and {\cite[Proposition 3.6]{CoskunHuizengaExistence}} in the $\F_e$ case] \label{thm:prioritaryProps}
Let $X = \P^2$ or $\F_e$ and let $\bv \in K(X)$ be a Chern character of positive rank.
\begin{enumerate}
\item  If $\Delta(\bv)\geq 0$, then the stack $\cP_L(\bv)$ is nonempty. \item The stack $\cP_L(\bv)$ is irreducible. 
\item If $r(\bv) \geq 2$, then the general sheaf in $\cP_L(\bv)$ is a vector bundle.
\end{enumerate}
If there is a semistable sheaf of character $\bv$, parts (2) and (3) also hold for the moduli space $M(\bv)$.
\end{theorem}

On the other hand, the nonemptinesss of the moduli spaces $M(\bv)$ is a very interesting question; for a detailed study of this question we refer the reader to the work of Dr\'ezet and Le Potier \cite{DLP,LePotier} if $X=\P^2$ or to \cite{CoskunHuizengaExistence} if $X=\F_e$.  In this paper we will mostly assume from the outset that we are working with the character $\bv$ of a stable sheaf, so the precise classification is not so important.  Of course it is true on any surface that if $\bv$ has positive rank and $\Delta(\bv) \gg 0$ then $M(\bv)$ is nonempty; see for example \cite{HuybrechtsLehn} or \cite{OGrady}.

\subsection{Cohomology of general sheaves}\label{sec:BN-prelim} We will make frequent use of Weak Brill-Noether theorems for $\PP^2$ and $\FF_e$. These theorems describe the cohomology of a general stable or prioritary vector bundle. The case of $\PP^2$ is a theorem of G\"ottsche and Hirschowitz \cite{GottscheHirschowitz}, and the case of $\FF_e$ is a theorem of Coskun and Huizenga \cite{CoskunHuizengaBN}.

\begin{theorem}[Weak Brill-Noether for $\PP^2$]\label{thm:weakBNP2}
Let $\bv$ be the Chern character of an $L$-prioritary vector bundle on $\PP^2$ with $\Delta(\bv) \geq 0$. Then the general $V \in \mathcal{P}_L(\bv)$ has at most one nonzero cohomology group.\end{theorem}

The following statement is weaker than the full Weak Brill-Noether theorem for $\FF_e$, but it is sufficient for our purposes.

\begin{theorem}[Weak Brill Noether for $\FF_e$]\label{thm:weakBNFe}
Let $\bv$ be the Chern character of an $L$-prioritary sheaf on $\FF_e$ satisfying $\nu(\bv) \cdot F \geq -1$ and $\Delta(\bv) \geq 0$. If $\nu(\bv) \cdot E \geq -1$, then the general $V \in \mathcal{P}_L(\bv)$ has at most one nonzero cohomology group, and furthermore, $H^2(\FF_e,V) = 0$. Conversely, if $\chi(\bv) \geq 0$, then the general sheaf in $\mathcal{P}_L(\bv)$ has at most one nonzero cohomology group if and only if $\nu(\bv) \cdot E \geq -1$.
\end{theorem}

\subsection{Globally generated characters}\label{sec:gg-prelim} In contrast to ampleness, global generation is not an open property in families: there can be special globally generated bundles that are globally generated because they have ``extra sections'' coming from nonzero higher cohomology (see \cite[Example 2.13]{CoskunHuizengaBN}).
However, in families of sheaves with no higher cohomology, global generation \emph{is} an open property.

Suppose $\bv\in K(X)$ is a character with $\Delta(\bv) \geq 0$ such that the general sheaf in $\cP_L(\bv)$ is a globally generated  vector bundle.  Since the restriction of a globally generated bundle to a curve is globally generated, it follows that $\nu(\bv)$ is nef.  By the Weak Brill-Noether theorems, the general $V\in \cP_L(\bv)$ has only one nonzero cohomology group, which must be $H^0(V)$ since $V$ is globally generated.  

The characters $\bv\in K(X)$ such that the general bundle $V\in \cP_L(\bv)$ is globally generated were fully classifed by Bertram, Goller, and Johnson \cite{BertramGollerJohnson} and Coskun and Huizenga \cite{CoskunHuizengaBN} in the case of $\PP^2$, and by Coskun and Huizenga \cite{CoskunHuizengaBN} in the case of $\FF_e$.  We will not need the full strength of the classification, so we give a simple criterion for global generation at the end of this section.

\begin{theorem}[\protect{\cite[Cor. 5.3]{CoskunHuizengaBN}}]\label{thm:CH_globalgenerationP2}
Let $\bv\in K(\P^2)$ be a Chern with $\Delta(\bv) \geq 0$ and $\rk(\bv) \geq 2$. Then the general $V \in \mathcal{P}_L(\bv)$ is globally generated if and only if one of the following holds.
\begin{enumerate}
        \item We have $\mu(\bv) = 0$ and $\bv = \rk(\bv) \ch \OO$.
        \item We have $\mu(\bv) > 0$ and $\chi(\bv(-1)) \geq 0$.
        \item We have $\mu(\bv) > 0$, $\chi(\bv(-1)) < 0$, and $\chi(\bv) \geq \rk(\bv) + 2$.
        \item We have $\mu(\bv) > 0$, $\chi(\bv(-1)) < 0$, $\chi(\bv) = \rk(\bv) + 1$, and $\bv = (\rk(\bv) + 1) \ch \OO - \ch \OO(-2)$.
\end{enumerate}
\end{theorem}

\begin{theorem}[\protect{\cite[Thm. 5.1]{CoskunHuizengaBN}}]\label{thm:CH_globalgenerationFe}
Let $\bv\in K(\F_e)$ be a Chern character where $e \geq 1$, $\Delta(\bv) \geq 0$, and $\rk(\bv) \geq 2$. Suppose $\nu(\bv)$ is nef.  Then the general $V \in \mathcal{P}_L(\bv)$ is globally generated if and only if one of the following holds.
\begin{enumerate}
        \item We have $\nu(\bv) \cdot F = 0$, and there exist integers $a,m>0$ such that
        \[
        \bv = (\rk(\bv) - m)(\ch\OO(aF)) + m\ch (\OO((a+1)F)).
        \]
        \item We have $\nu(\bv) \cdot F > 0$ and $\chi(\nu(-F)) \geq 0$.
        \item We have $\nu(\bv) \cdot F >0$, $\chi(\nu(-F)) < 0$, and $\chi(\bv) \geq \rk(\bv) + 2$.
        \item We have $e=1$, $\nu(\bv) \cdot F > 0$, $\chi(\bv(-F)) < 0$, $\chi(\bv) = \rk(\bv) + 1$, and
        \[
        \bv = (\rk(\bv) + 1)(\ch (\OO)) - \ch(\OO(-2E-2F)).
        \]
\end{enumerate}
\end{theorem}

The statement for $\FF_0$ is slightly different.

\begin{theorem}[\protect{\cite[Thm. 5.2]{CoskunHuizengaBN}}]\label{thm:CH_globalgenerationF0}
Let $\bv\in K(\F_0)$ be a Chern character with $\Delta(\bv) \geq 0$ and $\rk(\bv) \geq 2$.  Suppose  $\nu(\bv)$ is nef. Then the general $V \in \mathcal{P}_L(\bv)$ is globally generated if and only if one of the following holds.
\begin{enumerate}
        \item We have $\nu(\bv) \cdot E = 0$ or $\nu(\bv) \cdot F = 0$ and there are integers $a,m \geq 0$ such that
\[
\bv = (\rk(\bv)-m)\ch \OO(aE) + m\ch \OO((a+1)E)
\]
or
\[
\bv = (\rk(\bv)-m)\ch \OO(aF) + m\ch \OO((a+1)F).
\]
\item We have $\nu(\bv) \cdot E >0$, $\nu(\bv) \cdot F >0$, but $\chi(\bv(-E)) \geq 0$ or $\chi(\bv(-F)) \geq 0$.
\item We have $\nu(\bv) \cdot E>0$, $\nu(\bv) \cdot F > 0$, $\chi(\bv(-E)) < 0$, $\chi(\bv(-F)) < 0$, and $\chi(\bv) \geq \rk(\bv) + 2$.
\end{enumerate}
\end{theorem}

We can also give the following criterion for global generation which uniformly combines part (2) of the three preceding results.

\begin{corollary}\label{cor:ggcriterion}
Let $X = \P^2$ or $\F_e$ and let $\bv\in K(X)$ be a Chern character with $\Delta(\bv) \geq 0$ and $\rk(\bv) \geq 2$. Suppose $\nu(\bv)$ is big and nef.  If $$\chi(\bv(-L)) \geq 0,$$ then the general $V\in \mathcal{P}_L(\bv)$ is globally generated.
\end{corollary}

\section{Obstructions to ampleness}\label{sec:notample}
The goal of this section is to explain some precise restrictions on the Chern characters of stable ample bundles.  The results of the section are summarized in the next proposition, the proof of which will occupy the rest of the section.
\begin{proposition}\label{prop:obstruction}
Let $X=\P^2$ or $\F_e$ and let $V$ be an ample stable vector bundle on $X$ of rank $r\geq 2$.
\begin{enumerate}
        \item If $X = \PP^2$, then $\nu(V) \cdot H > 1 + \frac{1}{r}$ or $V\cong T_{\P^2}$.
\item If $X = \FF_0$, then $\nu(V)\cdot F > 1$ and $\nu(V)\cdot E > 1$.
        \item If $X = \FF_e$ with $e\geq 1$, then $\nu(V) \cdot F > 1$ and $\nu(V) \cdot E \geq 1$.
\end{enumerate}  
\end{proposition}

An easy necessary property of ample bundles $V$ on $\PP^2$ and $\FF_e$ is that the restriction of $V$ to any smooth rational curve $D$ must have degree at least equal to its rank. Suppose otherwise: then $V|_D$ splits as a direct sum $\OO_{\PP^1}(m_1) \oplus \cdots \oplus \OO_{\PP^1}(m_r)$ with $r = \rk V > \sum m_i$. Thus for some $i$, we have $m_i \leq 0$. In particular, $V|_D$ has a quotient that is not ample. We record the most important cases of this fact in the following proposition.

\begin{proposition}\label{prop:easy_obstruction}
Let $V$ be an ample vector bundle on the surface $X = \PP^2$ or $X = \FF_e$.
\begin{enumerate}
        \item If $X = \PP^2$, then $\mu(V) \geq 1$.
        \item If $X = \FF_e$, then $\nu(V) \cdot E \geq 1$ and $\nu(E) \cdot F \geq 1$.
\end{enumerate}
\end{proposition}

The above can be sharpened for both $\PP^2$ and $\FF_e$ if $V$ is assumed to be stable of rank at least $2$.  The next result generalizes \cite[Proposition 3.1]{LePotierAmple} to arbitrary rank.

\begin{proposition}\label{prop:converse_notample_p2}
Let $V$ be an ample stable bundle on $\PP^2$ of rank $r \geq 2$, and suppose  $\mu(V) \leq 1 + \frac{1}{r}$. Then $V \cong T_{\PP^2}$.
\end{proposition}
\begin{proof}

Let $\PP^{2*}$ denote the dual projective plane, and let $$\Sigma = \{(p,\ell):p\in \ell\} \subset \P^2 \times \PP^{2*}$$ be the universal line. Let $\pi_1:\P^2\times \PP^{2*}\to \P^2$ and $\pi_2:\P^2\times \PP^{2*}\to \PP^{2*}$ be the projections.  Since $V$ is ample, we know $V|_\ell$ is ample for every line $\ell$, and therefore $H^1(V(-2)|_\ell) = 0$.  By Grauert's theorem $$R^1\pi_{2*}(\pi_1^*(V(-2))\te \OO_\Sigma) = 0.$$ 
The structure sheaf $\OO_\Sigma$ fits into the restriction sequence $$0\to I_\Sigma \to \OO_{\P^2\times \PP^{2*}} \to \OO_\Sigma\to 0.$$ The universal line $\Sigma$ is defined by a bihomogeneous form of type $(1,1)$, so its ideal sheaf is $$I_\Sigma \cong \OO_{\P^2}(-1) \boxtimes \OO_{\P^{2*}}(-1).$$  Tensor the resolution of $\OO_\Sigma$ by $\pi_1^*(V(-2))$ and apply $\pi_{2\ast}$ to find that we must have a surjection $$R^1\pi_{2*}(\pi_1^*(V(-3)))\te \OO_{\PP^{2*}}(-1) \to R^1 \pi_{2*}(\pi_1^*(V(-2)))\te \OO_{\PP^{2*}} \to 0,$$ which can be computed to be a map $$A: H^1(V(-3))\te \OO_{\PP^{2*}}(-1)\to H^1(V(-2))\te \OO_{\PP^{2*}} \to 0.$$ Since the slope of $V$ satisfies $1\leq \mu(V) \leq 1+\frac{1}{r}$, stability implies that the only possible nonzero cohomology group of $V(-3)$ and $V(-2)$ is $H^1$.  From the exact sequence $$0\to V(-3)\to V(-2)\to V(-2)|_{\ell}\to 0$$ we get $\chi(V(-2)) = \chi(V(-3))+\chi(V(-2)|_\ell)$.  Let $\epsilon = \chi(V(-2)|_\ell)$  and  $m = -\chi(V(-2))$. Then the map $A$ is a surjective map $$A: \OO_{\PP^{2*}}(-1)^{m+\epsilon} \to \OO_{\PP^{2*}}^{m}.$$

\emph{Case 1: $\mu(V) = 1 + \frac{1}{r}$.} In this case, $\epsilon = 1$ and $A$ becomes a surjective map $$A: \OO_{\PP^{2*}}(-1)^{m+1} \to \OO_{\PP^{2*}}^m.$$ Then the kernel must be the line bundle $\OO_{\P^{2*}}(-m-1)$, and computing Euler characteristics gives a contradiction unless $m = 0$.  But if $m=0$, then $\chi(V(-2)) = 0$ so $$\Delta(V) = P\left(1+\frac{1}{r}-2\right).$$ For $r = 2$ this gives $\Delta(V) = \frac{3}{8}$ and $V = T_{\P^2}$. If $r \geq 3$, then it gives $0<\Delta(V) < \frac{3}{8}$.  There is no stable bundle with this discriminant since stable bundles with discriminant less than $1/2$ are exceptional and have discriminant $\frac{1}{2}(1-\frac{1}{r^2})$; see \cite{LePotier}.

\emph{Case 2: $\mu(V) = 1$.}  In this case, $\epsilon = 0$ and $A$ is a surjective map $$A : \OO_{\P^{2*}}(-1)^{m} \to \OO_{\P^{2*}}^m.$$ Then $A$ must be an isomorphism, which is only possible if $m = 0$.  As in the previous case, this implies $\Delta(V) = 0$.  Since $V$ has rank $2$ and is stable, this is impossible. \end{proof}

A similar result holds on Hirzebruch surfaces as well. Of course, the tangent bundle $T_{\FF_e}$ is not ample \cite{Mori}. 

\begin{proposition}\label{prop:converse_notample_fe}
Let $V$ be a stable ample bundle on $\FF_e$, $e \geq 0$. If $\nu(V) \cdot F = 1$, then $V$ is a line bundle.
\end{proposition}
\begin{proof}
Suppose $V \in M(\bv)$ is general and that $V$ is ample.  If $\ell \in |F|$ is any fiber, then $V|_\ell$ is ample, hence a direct sum of line bundles $V|_\ell \cong \bigoplus \OO_\ell(1)$. We can now argue as in \cite[Prop. 5.4]{CoskunHuizengaBN}: we have $\pi^\ast(\pi_\ast(V|_E)) \cong V$, and $V|_E$ is a balanced direct sum of line bundles by \cite[Prop. 3.6]{CoskunHuizengaBN}. Since $V$ is stable, it must be a line bundle, and then every $V\in M(\bv)$ is a line bundle.
\end{proof}

\section{Asymptotic ampleness}\label{sec:asymptotics}
In this section we perform our study of ampleness for Chern characters $\bv$ with fixed slope and discriminant but sufficiently large rank.  Under this assumption, the results of Section \ref{sec:notample} provide the only obstructions to ampleness. Recall that $H$ denotes the minimal ample polarization.

\begin{theorem}\label{thm:asymptoticample}
Let $X = \P^2$ or $\F_e$ and let $\bv \in K(X)$ be the Chern character of a stable vector bundle. Suppose either
\begin{enumerate}
\item $X = \P^2$ and $\nu(\bv)\cdot H > 1$, or

\item $X = \F_0$ and $\nu(\bv)\cdot F > 1$ and $\nu(\bv) \cdot E > 1$, or  

\item $X = \F_e$ with $e\geq 1$ and $\nu(\bv) \cdot F > 1$ and $\nu(\bv) \cdot E \geq 1$. 
\end{enumerate}
If $n$ is a sufficiently large integer and $V\in M_{H}(n\bv)$ is general, then $V$ is ample.
\end{theorem}

Also see Remark \ref{rem-effective} for an explicit effective bound on $n$.

\begin{remark} The required inequalities in Theorem \ref{thm:asymptoticample} can be more compactly rephrased as the condition that $\nu(\bv)-H$ is big and nef.
\end{remark}

\begin{proof}[Proof of Theorem \ref{thm:asymptoticample}]
Recall that we write $L  =H$ if $X = \P^2$ and $L = F$ if $X =\F_e$.
The idea of the proof is to show that the (irreducible) stack $\cP_L(n\bv)$ of $L$-prioritary sheaves contains a quotient $V$ of an obviously ample bundle.  Then $V$ is ample, so the general prioritary sheaf is ample, and since $M(n\bv)\subset \cP_L(n\bv)$ is an open dense substack, it follows that the general $V\in M(n\bv)$ is ample.  The proof proceeds in several steps.
 
\emph{Step 1: Normalization of $\nu(\bv)$.}  Without loss of generality, we may  assume that $1 < \nu(\bv)\cdot L \leq 2$.  Indeed, we can find a nef line bundle $N$ such that $\bv(-N)$ satisfies these inequalities (this is clear if $X=\P^2$, and if $X = \F_e$ we can use a bundle of the form $N = \OO(mE+meF)$, noting $N\cdot E = 0$ and $N\cdot F = m$).  Then if a bundle $W\in M(n\bv(-N))$  is ample, its twist $W(N) \in M(n\bv)$ is ample  as well.  In what follows, we assume $\nu(\bv)$ satisfies this inequality.  

\emph{Step 2: The character $\bu$.}  Fix an integer $s\geq 2$, and let $r = \rk(\bv)$.  We wish to show that if $n$ is sufficiently large, then there is some $V\in \cP_L(n\bv)$ that is a quotient of $\OO_X(H)^{nr+s}$.  Then there will be an exact sequence $$0\to U \fto\phi \OO_X(H)^{nr+s} \to V\to 0.$$  To construct $V$, we reverse this process by starting from a bundle $U$ and taking the cokernel of a general map $\phi:U\to \OO_X(H)^{nr+s}$. Let $\bu = (nr+s) \ch(\OO_X(H)) -n \bv$ be the Chern character of the hypothetical kernel $U$.  Note that $\bu$ depends on $s$ and $n$; the choice of $s$ is unimportant as long as $s \geq 2$, but allowing $n$ to become large is crucial for the remainder of the proof.

\emph{Step 3: $L$-prioritary vector bundles of character $\bu$ exist if $n\gg 0$.}    Clearly $\bu$ has rank $s\geq 2$.  Recall from Theorem \ref{thm:prioritaryProps} that $L$-prioritary vector bundles of character $\bu$ will exist if $\Delta(\bu) \geq 0$, so we just have to pick $n$ so that this is true.  Recall $$2s^2 \Delta(\bu) = c_1(\bu)^2 - 2s \ch_2(\bu),$$ and write $$B = \nu(\bv) - H,$$ recalling that $B$ is big and nef.  We compute $$c_1(\bu) = -nrB+sH.$$ Since $B^2 > 0$, we see that  $c_1(\bu)^2$ is a quadratic function of $n$ with a positive leading coefficient.  On the other hand, $\ch_2(\bu)$ is a linear function of $n$.  Therefore $\Delta(\bu) \geq 0$ for $n\gg 0$.

\emph{Step 4: Global generation of $\sHom(U,\OO_X(H))$.} Let $n\gg 0$ be large enough that there are stable bundles of character $\bu$, and let $U\in M(\bu)$ be general.  We show that $\sHom(U,\OO_X(H)) = U^*(H)$ is globally generated.  By the results of Section \ref{sec:gg-prelim}, this can be determined from the character $\bu$.  In fact, the full classification of globally generated characters is more than we need and we can use the criterion of Corollary \ref{cor:ggcriterion}.

Observe that $$\nu(U^*(H)) = \frac{nr}{s}(\nu(\bv)-H),$$ so $\nu(U^*(H))$ is nef.  Furthermore, we compute $\nu(U^*(H))\cdot L > 0.$  Thus $U^*(H)$ is big. It remains to show $$\chi(U^*(H-L)) \geq 0.$$  Since $\chi(\OO_X(-H-L))=0$, we get $$\chi(U^*(H-L)) = -n\chi(\bv^*(H-L)),$$ and we need to show $\chi(\bv^*(H-L))\leq 0$, whence it will follow that $U^\ast(H)$ is globally generated. This is where the choice of normalization of $\nu(\bv)$ in Step 1 becomes important.  

If $X = \P^2$, then $H=L$ and $\mu(\bv^*(H-L)) = \mu(\bv^*)$ is between $-2$ and $-1$.  A stable bundle with slope between $-2$ and $-1$ has neither sections nor $H^2$, so we find $\chi(\bv^*) \leq 0$.

Suppose instead that $X = \F_e$, and let $P$ be the Hilbert polynomial of $\OO_X$.  We claim that $P(\nu(\bv^*(H-L)))\leq 0$, and therefore that $\chi(\bv^*(H-L))) \leq 0$ by Riemann-Roch and the Bogomolov inequality.  Write $\nu(\bv^*(H-L)) = xE+yF$.  Since $1 < \nu(\bv) \cdot L \leq 2$ and $\nu(\bv) \cdot E \geq 1$, we have $-1\leq x < 0$ and $y-ex \leq -1$.    Recall that $$P(xE + yF) = (x+1)\left(y+1-\frac{1}{2}ex\right).$$ Then the first factor $x+1$ is nonnegative and the second factor $y+1 - \frac{1}{2}ex \leq y+1-ex \leq 0$ is nonpositive.  Therefore $P(xE+yF) \leq 0$.
 
\emph{Step 5: Construction of $V$}.  Let $\phi: U\to \OO_X(H)^{nr+s}$ be a general map.  Since $\sHom(U,\OO_X(H))$ is globally generated, a Bertini-type theorem \cite[Proposition 2.6]{HuizengaJAG} shows the cokernel $V$ is a vector bundle of character $n\bv$.

\emph{Step 6: $V$ is $L$-prioritary.}  To complete the proof, we need to show that $V$ is $L$-prioritary.  Since $V$ is a vector bundle, this amounts to showing that $\Ext^2(V,V(-L)) = 0$.  Twisting the defining sequence of $V$ by $\OO_X(-L)$ and applying $\Hom(V,-)$ gives a surjection  $$\Ext^2(V,\OO_X(H-L))^{nr+s}\to \Ext^2(V,V(-L))\to 0,$$ so it is enough to show that $\Ext^2(V,\OO_X(H-L)) = 0$.  Applying $\Hom(-,\OO_X(H-L))$ to the defining sequence of $V$ yields the exact sequence $$\Ext^1(U,\OO_X(H-L))\to \Ext^2(V,\OO_X(H-L)) \to \Ext^2(\OO_X(H)^{nr+s},\OO_X(H-L)),$$ and since $$\Ext^2(\OO_X(H),\OO_X(H-L)) = H^2(\OO_X(-L)) = 0,$$ it is enough to show that $$H^1(U^*(H-L)) = \Ext^1(U,\OO_X(H-L))  = 0.$$  We have already seen $\chi(U^*(H-L)) \geq 0$, so it suffices to show that $U^*(H-L)$ satisfies the weak Brill-Noether theorem.  Since $\nu(U^*(H))$ was shown to be nef, it is easy to see that the hypotheses of Theorems \ref{thm:weakBNP2} and \ref{thm:weakBNFe} are satisfied.
\end{proof}

\begin{remark}\label{rem-effective}
The bound on $n$ in Theorem \ref{thm:asymptoticample} can easily be made effective.  Taking $s=2$ and explicitly solving $\Delta(\bu) \geq 0$ for $n$ in Step 3, we see that it is enough to have
\[
n\geq \frac{4\Delta(\bv)}{rB^2}-\frac{2}{r}
\]
where $B = \nu(\bv) - H$ and $\bv$ is normalized as in Step 1.
\end{remark}

\begin{example}\label{example:Gieseker}
  We consider again the example of Gieseker, where a bundle $V_d$ is defined as a general cokernel \[
  0 \to \OO_{\PP^2}(-d)^2 \to \OO_{\PP^2}(-1)^4 \to V_{d} \to 0.
\]
The character $\bv_{d} = \ch(V_d)$ is given by $(\rk(\bv_{d}),c_1(\bv_{d}),\ch_2(\bv_{d})) = (2,2d-4,2-d^2)$. If $d \geq 4$, then we can analyze the characters $n\bv_d$ as in the theorem. In fact, we can often write ample bundles of character $n\bv_d$ as quotients of $\OO_{\PP^2}(1)^{2n+2}$. To make this effective, write an exact sequence
\[
  0 \to U \to \OO_{\PP^2}(1)^{2n+2} \to V \to 0.
\]
Assuming $\ch(V) = n\bv_d$, if we can show that $\chi(V^*) \leq 0$ and $\Delta(U) \geq 0$, then the argument in the theorem says that $V$ is locally free and prioritary. Computing Euler characteristics, we have
\[
\chi(V^*) = n(10-3d-d^2) < 0.
\]
Again arguing as in Step 3 of the theorem, we need $n$ large enough that $\Delta(U) \geq 0$. Calculating as in Remark \ref{rem-effective}, it suffices to have
\[
n > \frac{2(d-1)^2}{(d-3)^2} - 1.
\]
In particular, for $d \geq 12$, we may take $n=2$.

\end{example}

\section{Ample globally generated bundles}\label{sec:ample_globallygenerated}
Let $X = \P^2$ or $\F_e$  and let $\bv\in K(X)$ be a Chern character with $r(\bv)\geq 2$ such that the general sheaf $V\in M(\bv)$ is a globally generated vector bundle.   In this section we show that if the necessary inequalities of Proposition \ref{prop:obstruction} are satisfied, then $V$ is ample.  Note that $T_{\P^2}$ is both globally generated and ample, so we omit this case from the discussion.

\begin{theorem}\label{thm:main-gg}
Let $X = \P^2$ or $\F_e$ and let $\bv \in K(X)$ be a Chern character such that the general sheaf $V\in M(\bv)$ is a globally generated vector bundle.  Suppose either
\begin{enumerate}
\item $X = \P^2$ and $\mu(\bv) > 1+\frac{1}{r(\bv)}$, or

\item $X = \F_0$ and $\nu(\bv)\cdot F > 1$ and $\nu(\bv) \cdot E > 1$, or  

\item $X = \F_e$ with $e\geq 1$ and $\nu(\bv) \cdot F > 1$ and $\nu(\bv) \cdot E \geq 1$. 
\end{enumerate}
Then $V$ is ample.
\end{theorem}

Throughout the section, we let $\bv$ be a character satisfying the assumptions of Theorem \ref{thm:main-gg}, and we take $V\in M(\bv)$ to be a general globally generated vector bundle.  By Gieseker's Lemma, to show that $V$ is ample, we must show that for every irreducible curve $D\subset X$, the bundle $V|_D$ admits no trivial quotient.  Since $V$ is globally generated, this is equivalent to the vanishing $\Hom(V,\OO_D) = 0$.  Applying $\Hom(V,-)$ to the restriction sequence $$0\to \OO(-D)\to \OO\to \OO_D\to 0,$$ stability shows $\Hom(V,\OO) = 0$, and so there is an exact sequence $$0\to \Hom(V,\OO_D) \to \Ext^1(V,\OO(-D)).$$ We consider $\Ext^1(V,\OO(-D)).$  By Serre duality, $$\ext^1(V,\OO(-D)) = h^1(V(K_X+D)),$$ and we need to understand the cohomology of $V(K_X+D)$.  Our first result in this section shows that if $V\in M(\bv)$ is general, then this cohomology is determined by the Euler characteristic.

\begin{lemma}\label{lemma:general_vs_verygeneral}Let $\bv$ satisfy the assumptions of Theorem \ref{thm:main-gg}.
\begin{enumerate}
\item If $D\subset X$ is an irreducible curve class and $V\in M(\bv)$ is general (depending on $D$), then $V(K_X+D)$ has nonspecial cohomology.
\item If $V\in M(\bv)$ is general, then for any irreducible curve class $D\subset X$, the bundle $V(K_X+D)$ has nonspecial cohomology.
\end{enumerate}
\end{lemma}
\begin{proof}
(1) The first part of the lemma is a direct application of the Weak Brill-Noether theorems for $\P^2$ and $\F_e$.  Since $V$ is general in $M(\bv)$, the bundle $V(K_X+D)$ is general in $M(\bv(K_X+D))$.

If $X = \P^2$, then the conclusion comes from Theorem \ref{thm:weakBNP2}.

If $X = \F_e$, then we need to verify the slope assumptions in Theorem \ref{thm:weakBNFe}. Since $D$ is irreducible, we have $D\cdot F \geq 0$ and $D\cdot E \geq -e$.  Then  $$\nu(V(K_X+D))\cdot F = \nu(V)\cdot F +K_X\cdot F + D\cdot F > 1-2+0=-1$$ and $$\nu(V(K_X+D))\cdot E = \nu(V)\cdot E+K_X\cdot E+D\cdot E \geq 1+(e-2)-e=-1.$$ Thus the hypotheses of Theorem \ref{thm:weakBNFe} are satisfied.

(2) By (1), it is clear that for a \emph{very general} $V\in M(\bv)$ the bundle $V(K_X+D)$ has nonspecial cohomology for every irreducible curve class $D$.  But in fact, we will show that there is a finite list $D_1,\ldots,D_k$ of irreducible curve classes such that if $V\in M(\bv)$ is general and $V(K_X+D_i)$ has nonspecial cohomology for all $i$ then $V(K_X+D)$ has nonspecial cohomology for all irreducible curve classes $D$.

\emph{Case 1: $X=\P^2$}.  In this case, we take our list to be $H,2H,\ldots,kH$, where $k\geq 1$ is chosen such that $\chi(V(K_{\P^2}+kH)) \geq 0$.  We assume $V\in M(\bv)$ is general enough that its splitting type on a line $H$ is balanced.   Suppose $V(K_{\P^2}+kH)$ has no higher cohomology.  We have $$\mu(V(K_{\P^2}+(k+1)H))=\mu(V)-3+k+1\geq 0,$$ so $V(K_{\P^2}+(k+1)H)|_H$ has no higher cohomology.  Then from the restriction sequence $$0\to V(K_{\P^2}+kH) \to V(K_{\P^2}+(k+1)H)\to V(K_{\P^2}+(k+1)H)|_H\to 0,$$ we see that $V(K_{\P^2}+(k+1)H)$ has no higher cohomology.  Continuing inductively completes the argument.

\emph{Case 2: $X=\F_e$}.  We assume throughout that $e\geq 1$; a similar argument can be given when $e=0$.  In case $X = \F_e$, write $\nu(V) = xE +yF$.  Let $D= aE+bF$ be an irreducible divisor class other than $E$, so $0\leq ae \leq b$. Then \begin{align*}\frac{\chi(V(K_{\F_e}+D))}{\rk(V)} &= (x-2+a+1)(y-(e+2)+b+1-\frac{1}{2}e(x-2+a))-\Delta(V)\\
&\geq (x-1)(y-(e+2)+\frac{b}{2}+1-\frac{1}{2}e(x-2))-\Delta(V).
\end{align*}
Here $x = \nu(V)\cdot F> 1$, so there is an integer $B\geq 1$ such that if $b\geq B$, then $\chi(V(K_{\F_E}+D))\geq 0$.  

For our list of divisors, take the list of all irreducible divisors of the form $aE+bF$ with $b\leq B$.  Then any irreducible divisor not in this list can be obtained from a divisor of the form $aE+BF$ by repeatedly adding copies of $F$ or $E+eF$.  Let $C$ be one of the curve classes $F$ or $E+eF$.  Then curves of class $C$ are rational.  Furthermore, $-(K_X + C)$ is the class of a curve, so stable vector bundles are automatically $C$-prioritary.  Therefore  a general $V\in M(\bv)$ has balanced splitting type on a general curve of class $C$.  Let $D = aE+bF$ be an irreducible curve class with $b\geq B$, and inductively assume we know that $V(K_{\F_e}+D)$ has no higher cohomology.   By adjunction, $(K_{\F_e}+C)\cdot C = -2$, and $D^2 \geq 0$, so $$\nu(V(K_{\F_e}+D+C))\cdot C = \nu(V)\cdot C+D\cdot C - 2\geq -1.$$ Therefore $V(K_{\F_e}+D+C)|_C$ has no higher cohomology.  The restriction sequence $$0\to V(K_{\F_e} + D)\to V(K_{\F_e}+D+C) \to V(K_{\F_e}+D+C)|_C\to 0$$ then shows that $V(K_{\F_e}+D+C)$ has no higher cohomology.  
\end{proof}

The following immediate corollary shows that the vast majority of curve classes cannot provide an obstruction to ampleness.

\begin{corollary}\label{cor:big-curves}
Let $\bv$ satisfy the assumptions of Theorem \ref{thm:main-gg}, and let $V\in M(\bv)$ be general.  If $D$ is an irreducible curve such that $$\chi(V(K_X+D))\geq 0,$$ then $V|_D$ admits no trivial quotient.
\end{corollary}
\begin{proof}
The bundle $V(K_X+D)$ has no $H^1$ by Lemma \ref{lemma:general_vs_verygeneral}.  Following the discussion preceding the proof of Lemma \ref{lemma:general_vs_verygeneral}, we see that $\Hom(V,\OO_D)= 0$.
\end{proof}

Conversely, if the assumptions of Corollary \ref{cor:big-curves} are not satisfied, then families of stable sheaves on $X$ restrict to curves of class $D$ as nicely as possible.

\begin{corollary}\label{cor:bad_prioritary}
Let $\bv$ satisfy the assumptions of Theorem \ref{thm:main-gg}, and let $V\in M(\bv)$ be general.  If $D$ is an irreducible curve such that $$\chi(V(K_X+D))<0,$$ then $V$ is $D$-prioritary.
\end{corollary}
\begin{proof}
By Lemma \ref{lemma:general_vs_verygeneral}, $$\ext^2(V,\OO(-D)) = h^0(V(K_X+D))=0.$$ Since $V$ is globally generated, it fits into an exact sequence of the form $$0\to M\to \OO^a \to V\to 0.$$ Twisting by $\OO(-D)$ and applying $\Hom(V,-)$ shows that $\Ext^2(V,V(-D))$ is a quotient of $\Hom(V,\OO(-D))=0$.  Therefore $V$ is $D$-prioritary.
\end{proof}

We now determine the possible curve classes $D$ where Corollary \ref{cor:big-curves} does not apply.

\begin{lemma}\label{lem:bad_curves}
Let $\bv$ satisfy the assumptions of Theorem \ref{thm:main-gg}.  There are only finitely many irreducible curve classes $D$ with $$\chi(\bv(K_X+D)) < 0.$$ Furthermore, any such $D$ must have one of the following forms.
\begin{enumerate}
\item If $X = \P^2$, then $D=H$ or $D=2H$.
\item If $X = \F_0$, then $D$ is of the form $bE + F$ or $E + bF$ for some $b\geq 0$.
\item If $X = \F_1$, then $D$ is of the form $F$, $2E+2F$, or $E+bF$ for some  $b\geq 0.$
\item If $X = \F_e$ with $e\geq 2$, then $D$ is of the form $F$, $E$,  or $E+bF$ for some $b\geq e$.
\end{enumerate}
In every case, $D$ is a smooth rational curve.
\end{lemma}
\begin{proof}
Let $V\in M(\bv)$ be general.  We saw in the proof of Lemma \ref{lemma:general_vs_verygeneral} (2) that there are only finitely many irreducible curve classes $D$ with $\chi(\bv(K_X+D)) <0$.  If $K_X+D$ is an effective divisor, then since $V$ is globally generated the bundle $V(K_X+D)$ has a section.  By Lemma \ref{lemma:general_vs_verygeneral} this implies $\chi(V(K_X+D)) \geq 0$. Thus it suffices to list the irreducible curve classes $D$ such that $K_X+D$ is not effective.  These classes are of the above listed forms.
\end{proof}

For each of the curve classes $D$ in Lemma \ref{lem:bad_curves}, we need to show that a general $V\in M(\bv)$ does not admit a trivial quotient $\OO_C$ for any irreducible curve $C\in|D|$.  The key step is to compute the codimension of the locus of sheaves admitting a trivial quotient on a single curve.

\begin{lemma}
Let $\cV/S$ be a complete family of globally generated vector bundles on  $\P^1$ of slope $\mu \geq 1$ and rank $r$, parameterized by a smooth irreducible base $S$.  Let $Z\subset S$ be the subset $$Z = \{s\in S:\hom(\cV_s,\OO_{\P^1})>0\}.$$ Then every component of $Z$ has codimension at least $$r\mu-r+1.$$
\end{lemma}
\begin{proof}
We follow the notation from \cite[\S 15.4]{LePotier}.
Let $Z_k\subset Z$ be the locus where $\hom(\cV_s,\OO_{\P^1}) = k$, and let $s\in Z_k$ be any point.  Since $\cV_s$ is globally generated, there is a unique filtration $$0\to F_1 \to \cV_s \to \OO_{\P^1}^{k}\to 0.$$  Let $Y_k = \Flag(\cV/S; P_1,P_2)\fto{\pi} S$ be the relative flag scheme parameterizing filtrations of $\cV_s$ with these numerical invariants.  Since we are working on a curve and the family $\cV/S$ is complete, $Y_k$ is smooth and the natural map $T_sS\to \Ext_+^1(\cV_s,\cV_s)$ surjects.  Let $t\in Y_k$ be the point corresponding to the above exact sequence.  Near $t$, the map $\pi$ maps $Y_k$ isomorphically onto $Z_k$.  We have an exact sequence $$T_t Y_k \to T_s S \to \Ext^1_+(\cV_s,\cV_s)\to 0.$$ The tangent space $T_s Z_k \subset T_sS$ is the image of $T_t Y_k$, so the normal space to $Z_k$ at $s$ is identified with $$\Ext^1_+(\cV_s,\cV_s) \cong \Ext^1(F_1,\OO_{\P^1}^{k}).$$ Since $\Hom(F_1,\OO_{\P^2}) = 0$, we conclude that $Z_k$ has codimension $$-\chi(F_1,\OO_{\P^2}^{k}) = k(r\mu -r+k).$$ This is clearly minimized for $k=1$.
\end{proof}

Finally we show that the curve classes in Lemma \ref{lem:bad_curves} do not obstruct ampleness to complete the proof of Theorem \ref{thm:main-gg}.

\begin{proof}[Conclusion of the proof of Theorem \ref{thm:main-gg}]
Let $D$ be one of the finitely many curve classes in Lemma \ref{lem:bad_curves}. By Corollary \ref{cor:bad_prioritary}, we may assume that a general $V\in M(\bv)$ is $D$-prioritary.  Let $\cV/S$ be a complete family of globally generated, $D$-prioritary, stable bundles on $X$ of character $\bv$, parameterized by a smooth variety $S$.  Then for any irreducible $C\in |D|$, the family $\cV|_C/S$ is a complete family of globally generated bundles on $C\cong \P^1$.  For a fixed irreducible $C\in |D|$, the locus of $s\in S$ such that $\cV_s$ has $\OO_C$ as a quotient has codimension at least $$c:=r \nu(V)\cdot C -r + 1.$$ As $C$ varies in the family $|D|$ of dimension $d:=h^0(\OO_X(D)) -1$, the general $\cV_s$ will not admit a quotient of the form $\OO_C$ if $d<c$.  We compute $d$ and estimate $c$ case-by-case to verify this.

\emph{Case 1: $X= \P^2$}.  If $D = H$, we get $$c = r\mu(V)-r+1 > r\left(1+\frac{1}{r}\right)-r+1=2$$ so $c\geq 3$, but $d=2$.  If instead $D = 2H$ then $d=5$ and $$c = 2r\mu(V)-r+1> 2r\left(1+\frac{1}{r}\right) - r+1= r+3,$$ so $c\geq 6$ since $r\geq 2$.

\emph{Case 2: $X = \F_e$ and $D = F$}.  We have $d=1$ and $$c = r\nu(V)\cdot F-r+1 > r-r+1=1.$$

\emph{Case 3: $X = \F_e$ with $e\geq 1$ and $D = E$}.  This time, $d= 0$ and $$c= r\nu(V) \cdot E - r+1 \geq r-r+1 = 1.$$

\emph{Case 4: $X = \F_e$ and $D = E+bF$ with $b\geq e$}.  By Riemann-Roch, $$ d = 2b-e+1.$$ We have $$c= r\nu(V) \cdot (E+bF)-r+1> r+rb-r+1=rb+1 \geq 2b+1,$$ where to see  the strict inequality we consider the cases $e=0$ and $e\geq 1$ separately.  The case of $X = \F_0$ and $D = bE+F$ holds by symmetry.

\emph{Case 5: $X = \F_1$ and $D = 2E+2F$.}  We get $d = 5$ and $$c= r\nu(V)\cdot(2E+2F)-r+1>3r+1 \geq 7,$$ so again $d<c$.
\end{proof}

\bibliographystyle{plain}

\end{document}